\documentclass{amsart}

\usepackage{amsfonts}
\usepackage{color}
\usepackage{graphicx}

\newtheorem{thm}{Theorem}[section]

\newtheorem{lema}[thm]{Lemma}

\theoremstyle{definition}
\newtheorem{defn}[thm]{Definition}
\theoremstyle{remark}

\newtheorem{rem}[thm]{Remark}
\numberwithin{equation}{section}
\newcommand{\R}{\mathbb R}
\newcommand{\N}{\mathbb N}

\newcommand{\Jp}{\mathcal{J}_{1,p}}

\newcommand{\I}{\mathcal{I}}
\newcommand{\Jsp}{\mathcal{J}_{s,p}}
\newcommand{\Jk}{\mathcal{J}_{s_k,p}}
\newcommand{\Np}{\mathcal{N}_{1,p}}
\newcommand{\Nsp}{\mathcal{N}_{s,p}}
\def\glim{\mathop{\text{\normalfont $\Gamma-$lim}}}

 \newcommand{\ve}{\varepsilon}

\newcommand{\cd}{\rightharpoonup}

\begin{document}

\title{Stability of solutions for nonlocal problems}
	
\author[J. Fern\'andez Bonder]{Julian Fern\'andez Bonder}
\author[A. Salort]{Ariel Salort}

\address{Departamento de Matem\'atica, FCEyN - Universidad de Buenos Aires and
\hfill\break \indent IMAS - CONICET
\hfill\break \indent Ciudad Universitaria, Pabell\'on I (1428) Av. Cantilo s/n. \hfill\break \indent Buenos Aires, Argentina.}

\email[J. Fern\'andez Bonder]{jfbonder@dm.uba.ar}
\urladdr{http://mate.dm.uba.ar/~jfbonder}

\email[A. Salort]{asalort@dm.uba.ar}
\urladdr{http://mate.dm.uba.ar/~asalort}


%35R11: Fractional partial differential equations
%45G05: Singular nonlinear integral equations
%35B35: Stability

\begin{abstract}
In this paper we deal with the stability of solutions to fractional $p-$Laplace problems with nonlinear sources when the fractional parameter $s$ goes to 1. We prove a general convergence result for general weak solutions which is applied to study the convergence of ground state solutions of $p-$fractional problems in bounded and unbounded domains as $s$ goes to 1. Moreover, our result applies to treat the stability of $p-$fractional eigenvalues as $s$ goes to 1.
\end{abstract}

\maketitle

\section{Introduction}

In this paper we analyze the stability of solutions for fractional $p-$laplace equations when the fractional parameter goes to 1. This is the transition from nonlocal-to-local equations. This phenomena has been studied by several authors in the past, but as far as we are concerned, the problem for general solutions of nonlinear equations is, prior to this work, missing in the literature.

Our results rely deeply in the seminal papers of Bourgain-Brezis-Mironescu \cite{BBM,BBM2}, where the authors study the behavior of $p-$fractional energies as $s\uparrow 1$ (see also \cite{Ponce}).

In this line of research, in \cite{BPS} the authors dealt with the asymptotic behavior of eigenfunctions of the Dirichlet fractional $p-$Laplacian, i.e., a right hand side being a multiple of a $p-$power. Also, for the linear setting, in \cite{BHS} the authors studied the behavior as $s\uparrow 1$ of solutions of  the Poisson equation to its local counterpart. See also \cite{FBS} for a similar result in the context of fractional Sobolev-Orlicz spaces. Moreover,  in \cite{BS} the same task was done for ground state solutions of the fractional semilinear Schr\"odinger equation. Nevertheless, in the quasilinear case several technical difficulties arise and up to our knowledge this situation was not contemplated with general right hand side and that is the main aim of this manuscript.  

To be precise, we analyze the asymptotic behavior as $s\uparrow 1$ of any family of solutions of the problem
\begin{align} \label{eq.s}
\begin{cases}
(-\Delta_p)^s u = f(x,u) & \text{ in } \Omega\\
u=0 &\text{ in } \R^n \setminus \Omega,
\end{cases}
\end{align}
where the nonlinear term $f(x,u)$ is required to have a subcritical growth in the sense of the Sobolev embeddings, and prove that any accumulation point of the sequence of solutions is in fact a solution to the local limit problem
\begin{align} \label{eq.1}
\begin{cases}
-\Delta_p u = f(x,u) & \text{ in } \Omega\\
u=0 &\text{ on } \partial \Omega.
\end{cases}
\end{align}

As a consequence of our result, under appropriate further structural assumptions on $f$, we prove that any accumulation point of a sequence of ground state solutions to the fractional Schr\"odinger equation
$$
\begin{cases}
(-\Delta_p)^s u + V(x) |u|^{p-2} u = f(x,u) & \text{ in } \Omega\\
u=0 &\text{ in } \R^n \setminus \Omega,
\end{cases}
$$
is a ground state solution to the corresponding local Schr\"odinger equation
$$
\begin{cases}
-\Delta_p u + V(x) |u|^{p-2} u = f(x,u) & \text{ in } \Omega\\
u=0 &\text{ on } \partial\Omega.
\end{cases}
$$
For this problem, our method allow us to treat almost without changes the bounded and the unbounded domain  cases.

Finally, we apply our general result to deal with the eigenvalue problem
$$
\begin{cases}
(-\Delta_p)^s u = \lambda^s |u|^{p-2} u & \text{ in } \Omega\\
u=0 &\text{ in } \R^n \setminus \Omega,
\end{cases}
$$
and get some mild generalization of the results in \cite{BPS}.

\subsection*{Organization of the paper}
The rest of the paper is organized as follows. In Section 2, we collect some preliminaries needed in the course of the work. This material is well-known to experts (with the only possible exception of Lemma \ref{key.lema}), but we choose to include it in order to make the paper as self contained as possible. In Section 3 we prove our main result (Theorem \ref{main}) about the asymptotic behavior of any family of solutions to \eqref{eq.s} that has some uniform (in $s$) bound. Finally, in Sections 4 and 5 we apply the result in Section 3 to deal with the problem of ground state solutions to the nonlinear fractional Schr\"odinger equation and to the fractional eigenvalue problem respectively.

\section{Preliminaries}

\subsection{Fractional Sobolev spaces}
Given $s\in (0,1)$ and $1\le p<\infty$, for any $u\in L^1_{\rm loc}(\R^n)$ we define the Gagliardo $(s,p)-$seminorm as
$$
[u]_{s,p}^p := K(n,s,p) \iint_{\R^n\times\R^n} \frac{|u(x)-u(y)|^p}{|x-y|^{n+sp}}\, dxdy.
$$

The constant $K(n,s,p)$ is a normalizing constant that is defined as
$$
K(n,s,p):= (1-s) \mathcal{K}(n,p),
$$
where
$$
\mathcal{K}(n,p)^{-1} = \frac{1}{p}\int_{{\mathbb S}^{n-1}} w_n^p\, dS_w
$$
and ${\mathbb S}^{n-1}$ is the unit sphere in $\R^n$.

The main property of this constant is that, for any $u\in L^p(\R^n)$, one has that
$$
\lim_{s\uparrow 1} [u]_{s,p}^p = \|\nabla u\|_p^p,
$$
where the above limit is understood as equality if $u\in W^{1,p}(\R^n)$ and $\liminf_{s\uparrow 1} [u]_{s,p}^p=\infty$ otherwise. See \cite{BBM}.

Given $\Omega\subset \R^n$ an open set, we then define the fractional order Sobolev spaces as
\begin{align*}
W^{s,p}(\R^n) &:= \{u\in L^p(\R^n)\colon [u]_{s,p}<\infty\},\\
W^{s,p}_0(\Omega) &:= \{u\in W^{s,p}(\R^n)\colon u=0 \text{ a.e. in } \R^n\setminus\Omega\}.
\end{align*}

\begin{rem}
Another way to define the fractional order Sobolev spaces is to consider the Gagliardo energy in $\Omega$
$$
[u]_{s,p;\Omega}^p := K(n,s,p) \iint_{\Omega\times\Omega} \frac{|u(x)-u(y)|^p}{|x-y|^{n+sp}}\, dxdy
$$
and then define
$$
\widetilde W^{s,p}(\Omega) := \{u\in L^p(\Omega)\colon [u]_{s,p;\Omega}<\infty\},\quad \widetilde W^{s,p}_0(\Omega) := \overline{C^\infty_c(\Omega)},
$$
where the closure is taken with respect to the norm $\|u\|_{s,p;\Omega} = (\|u\|_{p;\Omega}^p + [u]_{s,p;\Omega}^p)^\frac{1}{p}$.

It is a known fact that, in general, $\widetilde W^{s,p}_0(\Omega)\subset  W^{s,p}_0(\Omega)$, but one has equality, for instance, if $\Omega$ is a Lipschitz domain. Moreover $\widetilde W^{s,p}_0(\Omega) =  W^{s,p}_0(\Omega) = \widetilde W^{s,p}(\Omega) = W^{s,p}(\R^n)|_\Omega$ if $0<s<\frac{1}{p}$. See \cite{DNPV}.
\end{rem}

It is convenient to introduce the notation, for $0<s\le 1\le p<\infty$, 
$$
\Jsp \colon W^{s,p}_0(\Omega)\to \R
$$ 
$$
\Jsp(u) = \begin{cases}
\frac{1}{p} [u]_{s,p}^p & \text{if } 0<s<1\\
\frac{1}{p} \|\nabla u\|_p^p & \text{if } s=1.
\end{cases}
$$

A fundamental fact that will be used throughout this paper is the following theorem due to \cite{BBM} (see also \cite{Ponce}).
\begin{thm}\label{thm.BBM}
For any sequence $0<s_k\to 1$, the sequence $\{\Jk\}_{k\in\N}$ $\Gamma-$converges to $\Jp$.
\end{thm}

Recall that $\Gamma-$convergence is the notion of convergence suitable for minimization problems and it is defined as follows
\begin{defn}
Let $X$ be a metric space and $F,F_k\colon X  \to \bar \R$. We say that $F_k$  $\Gamma-$converges to $F$ if for every $u\in X$ the following conditions are valid.

\begin{itemize}
\item[(i)] (lim inf inequality) For every sequence $\{u_k\}_{k\in\N}\subset X$ such that $u_k \to u$ in $X$, 
$$
F(u)\leq \liminf_{k\to\infty} F_k(u_k).
$$

\item[(ii)] (lim sup inequality). For every $u\in X$, there is a sequence $\{u_k\}_{k\in\N}\subset X$ converging to $u$ such that 
$$
F(u)\geq  \limsup_{k\to\infty} F_k(u_k).
$$
\end{itemize}
The functional $F$ is called the $\Gamma-$limit of the sequence $\{F_k\}_{k\in\N}$ and it is denoted by $F_k \stackrel{\Gamma}{\to} F$ and 
$$
	F=\glim_{k\to\infty} F_k.
$$
\end{defn}

Another well-known fact that will be used throughout is the Sobolev immersion theorem, a proof of which can be found, for instance, in \cite{DNPV}.
\begin{thm}\label{inmersion}
Assume that $\Omega$ has finite measure. Define the critical Sobolev exponent as
$$
p^*_s := \begin{cases}
\frac{np}{n-sp} & \text{if } sp<n\\
\infty & \text{otherwise}.
\end{cases}
$$
(we will denote $p^*_1 = p^*$).

Then, $W^{s,p}_0(\Omega)\subset L^q(\Omega)$ with compact inclusion for every $1\le q < p^*_s$.
\end{thm}

%\begin{rem} \label{cte.sobolev} By the results of \cite{BBM2}, it readily follows that, given $s_0\in (0,1)$, there exists a constant $C=C(n,p,q)$ such that, for $s\in [s_0,1)$,
%$$
%\|u\|_q \leq C [u]_{s,p}
%$$ 
%for $q\in [p,\tfrac{np}{n-s_0p}]$.
%\end{rem}

The following notation will be enforced.
\begin{defn}
Given $0<s\le 1\le p <\infty$, the (topological) dual space of $W^{s,p}_0(\Omega)$ will be denoted by $W^{-s,p'}(\Omega)$.
\end{defn}

\subsection{Weak solutions}

Recall that, for any $0<s\le 1< p<\infty$, the functional $\Jsp$ is Fr\'echet differentiable and $\Jsp'\colon W^{s,p}_0(\Omega)\to W^{-s,p'}(\Omega)$ is continuous and is given by
\begin{align*}
&\langle \Jsp'(u), v\rangle = \frac{K(n,s,p)}{2}\iint_{\R^n\times\R^n} \frac{|u(x)-u(y)|^{p-2}(u(x)-u(y))(v(x)-v(y))}{|x-y|^{n+sp}}\, dxdy,\\
&\langle \Jp'(u), v\rangle = \int_\Omega |\nabla u|^{p-2}\nabla u\cdot\nabla v\, dx.
\end{align*}

Therefore, for $0<s\le 1<p$ we define the \emph{fractional $p-$Laplace operator} as $(-\Delta_p)^s := \Jsp'$. Hence we say that $u\in W^{s,p}_0(\Omega)$ is a \emph{weak solution} of \eqref{eq.s} if
$$
\langle (-\Delta_p)^s u, v\rangle = \int_\Omega f(x,u)v\,dx
$$
for all $v\in W^{s,p}_0(\Omega)$. Similarly, $u\in W^{1,p}_0(\Omega)$ is a \emph{weak solution} of \eqref{eq.1} if
$$
\langle -\Delta_p u, v\rangle = \int_\Omega f(x,u)v\,dx
$$
for all $v\in W^{1,p}_0(\Omega)$.

It is worth of mention that this operator is \emph{monotone} in the sense that for any $u,v \in W^{s,p}_0(\Omega)$ it holds that
$$
0\leq \langle (-\Delta_p)^s u -  (-\Delta_p)^s v ,u-v \rangle.
$$
\begin{rem}
Although it will not be used in this work, the operator $(-\Delta_p)^s$ is in fact {\em strictly monotone}. This is a consequence of a well known inequality proved by \cite{Simon}
$$
(|a|^{p-2}a - |b|^{p-2}b)\cdot (a-b) \ge \begin{cases}
c |a-b|^p & \text{if } p\ge 2\\
c \frac{|a-b|^2}{(|a|+|b|)^{2-p}} & \text{if } 1<p<2,
\end{cases}
$$
for any $a, b\in \R^N$ ($N\in \N$), where the constant $c$ depends on $p$ and $N$.

This immediately implies that
$$
\langle (-\Delta_p)^s u -  (-\Delta_p)^s v ,u-v \rangle \ge \begin{cases}
c \Jsp(u-v) & \text{if } p\ge 2\\
c \frac{\Jsp(u-v)^\frac{2}{p}}{(\Jsp(u) + \Jsp(v))^\frac{2-p}{p}} & \text{if } 1<p<2,
\end{cases}
$$
for $0<s\le 1<p<\infty$, the constant $c$ depending only on $p$ and $n$.
\end{rem}

The next lemma gives us some uniform asymptotic development for the functionals $\Jsp$.

\begin{lema}\label{asymptotic.development}
Let $u\in W^{1,p}_0(\Omega)$ be fixed and for any $s\in (0,1]$, let $v_s\in W^{s,p}_0(\Omega)$ be such that $[v_s]_{s,p}<C$ for any $s\in (0,1]$. Then, for $t>0$,
$$
\Jsp(u+tv_s) = \Jsp(u) + t\langle (-\Delta_p)^s u,v_s\rangle + o(t),
$$
where $o(t)$ depends only on $C$.
\end{lema}

\begin{proof}
The proof is a direct consequence of the elementary estimate
$$
|a+tb|^p = \begin{cases}
|a|^p + tp|a|^{p-2}ab + O(t^2) & \text{if } p\ge 2\\
|a|^p + tp|a|^{p-2}ab + o(t^p) & \text{if } p< 2,
\end{cases}
$$
where $O(t^2)$ (or $o(t^p)$ respectively) is uniform in $|b|$.
\end{proof}

With the help of Lemma \ref{asymptotic.development} we can prove a key lemma that can be though as an extension of Theorem \ref{thm.BBM}.

\begin{lema} \label{key.lema}
Let $s_k\uparrow 1$ and $v_k\in W^{s_k,p}_0(\Omega)$ be such that $\sup_{k\in\N} [v_k]_{s_k,p}^p<\infty$. Assume, without loss of generality, that $v_k\to v$ strongly in $L^p(\Omega)$. Then, for every $u\in W^{1,p}_0(\Omega)$, we have
$$
\langle (-\Delta_p)^{s_k} u, v_k\rangle \to   \langle -\Delta_p u, v\rangle.
$$
\end{lema}

\begin{proof}
First, observe that from the results in \cite{BBM}, it follows that $v\in W^{1,p}_0(\Omega)$ and so everything is well defined.

Now, it is enough to show that
\begin{equation} \label{desig.1}
\langle -\Delta_p u, v\rangle  \leq    \liminf_{k\to\infty} \langle (-\Delta_p)^{s_k} u, v_k\rangle.
\end{equation}
In fact, if \eqref{desig.1} holds for every $u\in W^{1,p}_0(\Omega)$, then apply \eqref{desig.1} to $-u$ to get the reverse inequality.

Now, by a refinement of Section 3  in \cite{BBM} (see also  \cite{DPS},  \cite{FBS} or \cite{Ponce}), we have
$$
\Jp (u+tv) \leq \liminf_{k\to\infty} \Jk(u+tv_k).
$$
The previous expression together with \cite{BBM} gives that
$$
\Jp(u+tv) -\Jp(u) \leq \liminf_{k\to\infty} (\Jk(u+tv_k) - \Jk(u)).
$$

Applying now Lemma \ref{asymptotic.development}, we obtain
$$
\langle -\Delta_p u, v\rangle  + o(1) \leq \liminf_{k\to\infty} \langle (-\Delta_p)^{s_k} u, v_k\rangle + o(1),
$$
from where \eqref{desig.1} follows.
\end{proof}

\section{Stability of weak solutions}

In this section, we prove our main result on the convergence of solutions of problems \eqref{eq.s} to solutions of \eqref{eq.1}.

In this section we ask the nonlinearity $f$ to satisfy the following hypotheses:
\begin{enumerate}
\item[($f_1$)]   $f\colon\Omega\times  \R \to \R$ is a Carath\'eodory function, i.e. $f(\cdot,z)$ is measurable for any $z\in\R$  and $f(x,\cdot)$ is continuous a.e. $x\in \Omega$.

\item[($f_2$)] There exist a constant $C>0$ such that $|f(x,z)|\le C(1+|z|)^{q-1}$ for some $q\in [1,p^*)$.
\end{enumerate}

\begin{rem}
Observe that if $q\in [1,p^*)$, then there exists $s_0\in (0,1)$ such that $q<p^*_s$ for any $s\in [s_0,1)$.
\end{rem}

\begin{rem}
Hypotheses ($f_1$) and ($f_2$) are the natural requirements to define the notion of weak solutions for \eqref{eq.s} and \eqref{eq.1}.
\end{rem}

\begin{thm}\label{main}
Let $0<s_k\to 1$ and let $u_k\in W^{s_k,p}_0(\Omega)$ be a sequence of solutions of \eqref{eq.s} such that $\sup_{k\in\N} [u_k]_{s_k,p}^p<\infty$. Then, any accumulation point $u$ of the sequence $\{u_k\}_{k\in\N}$ in the $L^p(\Omega)-$topology verifies that $u\in W^{1,p}_0(\Omega)$ and it is a weak solution of \eqref{eq.1}.
\end{thm}

\begin{proof}
Assume that $u_k\to u$ in $L^p(\Omega)$. Then, since $\{u_k\}_{k\in\N}$ is uniformly bounded in $W^{s_k,p}_0(\Omega)$, by \cite{BBM} we obtain that $u\in W^{1,p}_0(\Omega)$. Passing to a subsequence, if necessary, we can also assume that $u_k\to u$ a.e. in $\Omega$.

On the other hand, if we define $\eta_k:= (-\Delta_p)^{s_k} u_k\in W^{-s_k,p'}(\Omega)\subset W^{-1,p'}(\Omega)$, then $\{\eta_k\}_{k\in\N}$ is bounded in $W^{-1,p'}(\Omega)$ and hence, up to a subsequence, there exists $\eta\in W^{-1,p'}(\Omega)$ such that $\eta_k\cd \eta$ weakly in $W^{-1,p'}(\Omega)$.

Since $u_k$ solves \eqref{eq.s}, for any  $v\in C_c^\infty(\Omega)$
$$
0 = \langle (-\Delta_p)^{s_k} u_k ,v \rangle    -\int_\Omega f(x,u_k)v\,dx
$$
using the convergences, taking the limit $k\to\infty$ we get
$$
0=\langle \eta, v \rangle -\int_\Omega f(x,u)v\,dx.
$$
We want to identify $\eta$, more precisely, we will prove that 
\begin{equation} \label{will.prove}
\langle \eta, v \rangle = \langle -\Delta_p u,v\rangle.
\end{equation}
For that purpose we use the monotonicity of the operator and the fact that $u_k$ is solution of \eqref{eq.s}, Indeed,
\begin{align*}
0&\leq \langle (-\Delta_p)^{s_k} u_k,u_k-v \rangle - \langle (-\Delta_p)^{s_k} v, u_k-v \rangle\\
&= 
\int_\Omega f(x,u_k) (u_k-v)\,dx  - \langle (-\Delta_p)^{s_k} v, u_k - v\rangle.
\end{align*}
Hence taking the limit $k\to\infty$ and using Lemma \ref{key.lema} one finds that
\begin{align*}
0&\leq \int_\Omega f(x,u)(u-v) -   \langle -\Delta_p v,u-v \rangle \\
&=\langle \eta, u-v \rangle -\langle -\Delta_p v,u-v \rangle.
\end{align*}
Consequently, if we take $v = u-tw$, $w\in W^{1,p}_0(\Omega)$ given and $t>0$, we obtain that
$$
0\leq \langle \eta, w \rangle - \langle -\Delta_p (u-tw), w\rangle
$$
taking $t\to 0^+$ gives that
$$
0\leq \langle \eta, w \rangle - \langle -\Delta_p u, w\rangle.
$$
From this it is easy to see that \eqref{will.prove} holds and the proof concludes.
\end{proof}

\begin{rem}\label{remark.clave}
The results of Theorem \ref{main} can be easily improved by considering
$$
\begin{cases}
(-\Delta_p)^s u = f_s(x,u) & \text{ in } \Omega\\
u=0 &\text{ in } \R^n \setminus \Omega,
\end{cases}
$$
where $f_s(x,z)\to f(x,z)$ uniformly on compact sets of $z\in\R$. The proof of this fact is completely analogous to that of Theorem \ref{main} and is left to the reader.
\end{rem}

\section{Convergence of ground states}
This section is devoted to study the behavior of ground state (or least-energy) solutions of the nonlocal Schr\"odinger problem
\begin{equation} \label{eq.s.sch}
\begin{cases}
(-\Delta_p)^s u + V(x) |u|^{p-2}u = f(x,u) & \text{ in } \Omega\\
u=0 &\text{ in } \R^n \setminus \Omega,
\end{cases}
\end{equation}
to the limit problem
\begin{align} \label{eq.1.sch}
\begin{cases}
-\Delta_p u + V(x) |u|^{p-2}u = f(x,u) & \text{ in } \Omega\\
u=0 &\text{ on } \partial \Omega.
\end{cases}
\end{align}

In the semilinear case, that is when $p=2$, this problem was addressed in \cite{BS}. The methods used in that paper heavily use the linearity of the operator. Here we show how applying the results in the previous section, we can extend the main theorem in \cite{BS} to the more general quasilinear case. Moreover, in \cite{BS} only the bounded domain case is considered. Here we will consider both the bounded and unbounded domain cases.

Recall that ground state solutions are minimizers of the energy functional
$$
\I_{s,p}(u) := \Jsp(u) + \frac{1}{p}\int_\Omega |u|^p V(x)\, dx - \int_\Omega F(x,u)\, dx,
$$
restricted to the so-called {\em Nehari manifold}
$$
\Nsp:=\{ u \in W^{s,p}_0(\Omega) \setminus \{0\} \colon \langle \I_{s,p}'(u),u \rangle=0\}.
$$
Here $F(x,z) = \int_0^z f(x,\tau)\, d\tau$ is the primitive of $f$.

From Theorem \ref{main}, we know that if $\{u_s\}_{s\in(0,1)}$ is a sequence of solutions to \eqref{eq.s.sch}, then any accumulation point (in $L^p(\Omega)$) $u$ is a solution to \eqref{eq.1.sch}. The natural question now is to see if $u$ is a ground state solution whenever the sequence $\{u_s\}_{s\in(0,1)}$ are also ground states.

\subsection{The bounded domain case}
In this subsection, we assume that $\Omega$ is bounded. On the nonlinearity $f$, besides  ($f_1$) and ($f_2$) will be assumed to fulfill the following further structural hypothesis that are standard when consider ground state solutions in nonlinear problems (see, for instance, \cite{Szulkin})

\medskip

\begin{enumerate}
\item[($f_3$)]  $F(x,z)|z|^{-p} \to \infty$ as $|z|\to\infty$ uniformly with respect to  $x\in \Omega$.

\item[($f_4$)] $f(x,z)=o(|z|^{p-1})$ as $z\to 0$ uniformly with respect to $x\in\Omega$.

\item[($f_5$)]   For almost every $x\in\Omega$
$$
\frac{f(x,z)}{|z|^{p-1}} \quad \text{ is strictly increasing on } (-\infty, 0)\cup(0,\infty).
$$

\item[($f_6$)] There exists $\mu>p$ such that
$$
\mu F(x,z)\le zf(x,z).
$$
\end{enumerate}

On the potential function $V$ we assume:

\begin{enumerate}
\item[($V_1$)] $0\le V\in L^r(\Omega)$ for some $r>\frac{n}{p}$.
\end{enumerate}

\begin{rem}
From hypotheses ($f_1$) and ($f_4$) it follows that for every $\ve>0$ there is $C_\ve>0$ such that 
\begin{equation}   \label{F5} 
|f(x,z)|\leq \ve |z|^{p-1} + C_\ve |z|^{q-1}
\end{equation}
for every $z\in\R$ and a.e. $x\in\Omega$. 
\end{rem}

\begin{rem}
Hypothesis ($f_6$) is the fundamental structural hypothesis needed in variational arguments for the existence of ground state solutions to \eqref{eq.s.sch} and \eqref{eq.1.sch}. This is the well-known {\em Ambrosetti-Rabinowitz condition} first introduced in \cite{AR}. See also \cite{Yu} for the quasilinear case in the local setting.
\end{rem}

It is well-known that under ($f_1$)--($f_6$) a ground states actually exists. Indeed, for every $s\in (0,1]$ a ground state solution is a mountain pass solution and hence it fulfills the formula
\begin{equation} \label{caracterizacion}
\I_{s,p}(u_s)= \inf_{v\in W^{s,p}_0(\Omega)\setminus \{0\}} \sup_{t\in [0,1]} \I_{s,p}(tv)>0.
\end{equation}

Moreover, the Nehari manifold $\Nsp$ is homeomorphic to the unit sphere $\mathcal{S}_{s,p}$ in $W^{s,p}_0(\Omega)$ with homeomorphism is given by
$$
m_s(u)=t^s_u u
$$
where $t^s_u$ is the unique positive number such that $t^s_u u \in \Nsp$.

See, for instance, \cite{Jabri} for a good introduction to this subject and a proof of all of these facts in the local setting. 

\medskip

It will be convenient to introduce the notation for $s\in (0,1)$
$$
\|u\|_{s,p,V} = \|u\|_s := \left([u]_{s,p}^p + \int_\Omega |u|^p V(x)\, dx\right)^\frac{1}{p}
$$
and
$$
\|u\|_{1,p,V} :=  \left(\|\nabla u\|_p^p + \int_\Omega |u|^p V(x)\, dx\right)^\frac{1}{p}
$$

Observe that with this notation, we have
$$
\I_{s,p}(u) = \frac{1}{p}\|u\|_s^p - \int_\Omega F(x,u)\, dx.
$$

Our first lemma shows that the mountain pass levels of the ground state solutions of \eqref{eq.s.sch} are uniformly bounded.

\begin{lema} \label{cota.limsup}
Under the above notation and assumptions, it holds that $c_1 \geq \limsup_{s\uparrow 1} c_s$.
\end{lema}
\begin{proof}
Let   $u\in \Np$. Then  have that there exists $t_s>0$ such that $t_s u \in \Nsp$ and hence
\begin{align*}
\limsup_{s\uparrow 1} c_s &\leq \limsup_{s\uparrow 1} \I_{s,p}(t_s u)\\
&= \limsup_{s\uparrow 1} \left(\I_{s,p}(t_s u) - \frac1p \langle \I_{s,p}'(t_s u), t_s u\rangle \right)\\
&= \frac1p \left( \int_\Omega f(x,t_s u) t_s u - p F(x,t_s u)\, dx \right).
\end{align*}

In view of ($f_6$) we have that
\begin{align*}
\|u\|_s^p &= \frac{1}{t_s^p} \int_\Omega f(x,t_s u) t_s u\,dx\\ 
&\geq \frac{p}{t_s^p}\int_\Omega F(x,t_s u)\,dx = p\int_\Omega \frac{F(x,t_s u) }{t_s^p |u|^p} |u|^p\,dx.
\end{align*}
Since the left hand side of the inequality above tends to $\|u\|_{1,p, V}^p <\infty$, ($f_3$) implies that $\{t_s\}_{s}$ is bounded.

Let $t_0\ge 0$ be any accumulation point of $\{t_s\}_s$ and $\{t_{s_k}\}_{k\in\N}\subset \{t_s\}_s$ be such that $t_{s_k}\to t_0$ as $k\to\infty$. Let us see that in fact $t_0\neq 0$. In view of the Nehari identity
$$
\|u\|_{s_k}= \int_\Omega \frac{f(x,t_{s_k}u)}{t_{s_k}^{p-1} |u|^{p-1}}  |u|^p\,dx
$$
But again, the left hand side tends to $\|u\|_{1,p, V}^p>0$ and so, by ($f_4$) we get that $t_0>0$.

Furthermore, in view of \eqref{F5} we have that
$$
|f(x,t_{s_k} u) t_{s_k} u|\leq  \ve |t_{s_k}u|^p + C_\ve |t_{s_k} u|^q \leq C (|u|^p + |u|^q),
$$
where $C>0$ is independent of $k$, then by using the dominated convergence theorem,
$$
\int_\Omega f(x,t_{s_k} u) t_{s_k} u \,dx \to \int_\Omega f(x,t_0 u) t_0 u\,dx.
$$

In view of the computations above, as $k\to\infty$ we get
$$
t_0^p \|u\|_{1,p, V}^p = \int_\Omega f(x,t_0 u) t_0 u\,dx,
$$
but since $u\in \Np$, we deduce that $t_0=1$ and then $t_s\uparrow 1$ as $s\uparrow 1$.

Moreover, due to ($f_6$) we can apply again the dominated convergence theorem in the limit as $s\uparrow 1$ in the integral $ \int_\Omega F(x,t_s u)\,dx$ giving that
\begin{align*}
\limsup_{s\uparrow 1}\I_{s,p}(t_s u) &= \limsup_{s\uparrow 1}\frac1p \left\{ \int_\Omega f(x,t_s u) t_s u - p F(x,t_s u)\,dx \right\}\\
&=\frac1p \left\{ \int_\Omega f(x,u) u - p F(x,u)\,dx \right\}\\
&=\I_{1,p}(u)
\end{align*}
and the proof concludes by taking infimum over $u\in \Np$.
\end{proof}

Our next lemma proves that any sequence of ground state solutions to \eqref{eq.s.sch} is uniformly bounded with respect to $s$ away from zero and infinity.

\begin{lema} \label{lema.acotado}
Let $u_s\in W^{s,p}_0(\Omega)$ be a  ground state   solution of \eqref{eq.s.sch} with $s\in (0,1)$. Then there exist two constants $0<c<C<\infty$ independent on $s$ such that
$$
c\le \|u_s\|_s\le C.
$$
\end{lema}

\begin{proof}
Let $u_s\in W^{s,p}_0(\Omega)$ be a ground state solution to \eqref{eq.s.sch}.

By Lemma \ref{cota.limsup}, there exists a constant $C>0$ independent of $s\in (0,1)$ such that
$$
\I_{s,p}(u_s) \le C.
$$
But, since $u_s\in \Np$, it follows that
\begin{align*}
C &\ge \I_{s,p}(u_s) - \frac{1}{\mu}\langle \I'_{s,p}(u_s), u_s\rangle \\
& = \left(\frac{1}{p} - \frac{1}{\mu}\right) \|u_s\|_s^p - \int_\Omega \left(F(x, u_s) - \frac{1}{\mu} f(x,u_s)u_s\right)\, dx\\
&\ge  \left(\frac{1}{p} - \frac{1}{\mu}\right) \|u_s\|_s^p,
\end{align*}
where we have used ($f_6$) in the last inequality. 

Hence, $\sup_{s\in (0,1)} \|u_s\|_s <\infty$.

For the lower bound, we simply observe that since $u_s\in \Nsp$, we have, by \eqref{F5},
$$
\|u_s\|_s^p = \int_\Omega f(x,u_s)u_s\, dx \le \ve \|u_s\|_p^p + C_\ve\|u_s\|_q^q.
$$
By Theorem \ref{inmersion}, it follows that
$$
\|u_s\|_s^p \le C\ve\|u_s\|_s^p  + C C_\ve \|u_s\|_s^q. 
$$
From this inequality the lower bound follows easily.
\end{proof}

We are now in position to prove the main result of the section.

\begin{thm}\label{thm.ground}
Let $s\in (0,1)$ and let $u_s\in W^{s,p}_0(\Omega)$ be a ground state solution of \eqref{eq.s.sch}. Then, any accumulation point $u$ of $\{u_s\}_s$ in the $L^p(\Omega)-$topology verifies that $u\in W^{1,p}_0(\Omega)$ and is a ground state solution of  \eqref{eq.1.sch}.
\end{thm}

\begin{proof} 

Let $s_k\uparrow 1$ be a sequence such that $u_{s_k}\to u$ in $L^p(\Omega)$.

From Lemma \ref{lema.acotado} we obtain that $\sup_{k\in\N} [u_{s_k}]_{s_k,p}<\infty$ and hence, by \cite{BBM}, $u \in W^{1,p}_0(\Omega)$ and $u_{s_k}\to u$ strongly in $L^r(\Omega)$ for any $1\le r<p^*$.

We claim that $u\neq 0$. Indeed, since $u_{s_k}\in \mathcal{N}_{s_k, p}$, we have that 
$$
0<c\le \|u_{s_k}\|_{s_k}^p = \int_\Omega f(x,u_{s_k}) u_{s_k}\,dx \to \int_\Omega f(x, u) u\, dx.
$$

Let us see now that $u$ is a ground state solution of \eqref{eq.1.sch}. Since $u_{s_k}\in \mathcal{N}_{s_k,p}$ and since, by Theorem \ref{main}, $u$ is a weak solution to \eqref{eq.1.sch}, we obtain that
\begin{align*}
\liminf_{k\to\infty} c_{s_k}&= \liminf_{k\to\infty} \I_{s_k,p}(u_{s_k}) = \liminf_{k\to\infty} ( \I_{s_k,p}(u_{s_k}) - \tfrac1p\langle \I_{s,p}' u_{s_k},u_{s_k} \rangle)\\
& = \frac{1}{p}\liminf_{k\to\infty} \left(  \int_\Omega f(x,u_{s_k})u_{s_k} -  p\int_\Omega F(x,u_{s_k})\,dx  \right)\\
&=\frac1p  \left(  \int_\Omega f(x,u)u -  p\int_\Omega F(x,u)\,dx  \right) = \I_{1,p}(u)\geq c_1.
\end{align*}
Therefore, as a consequence of Lemma \ref{cota.limsup} and the inequality above we get
$$
\liminf_{k\to\infty} c_{s_k} \geq c_1 \geq \limsup_{s\uparrow 1} c_{s} \geq \limsup_{k\to\infty} c_{s_k},
$$
resulting in $\lim_{k\to\infty} c_{s_k}=c_1 = \I_{1,p}(u)$.
\end{proof}

\subsection{The unbounded domain case}

In this subsection we consider $\Omega\subset\R^n$ to be a general unbounded continuous open and connected set (for instance, $\Omega=\R^n$).

There are several cases where the existence of a ground state for problems \eqref{eq.s.sch} and \eqref{eq.1.sch} is verified. As an example of those cases, in this subsection we consider the case where the source term $f(x,z)$ in addition to ($f_1$)--($f_6$) also verifies
\begin{enumerate}
\item[($f_7$)] $|f(x,z)|\le \omega(x) |z|^{q-2}z$ where $p<q<p^*$ and $\omega\in L^\infty(\Omega)\cap L^{p_1}(\Omega)$, where $p_s=np/(np-q(n-sp))$ for $0<s\le 1$.
\end{enumerate}

\begin{rem}
Observe that for $0<s<1$, one has that $p_1<p_s$. Hence, by interpolation, $\omega\in L^{p_s}(\Omega)$ for every $0<s\le 1$.
\end{rem}

It is proven in \cite{Yu} that under ($f_1$)--($f_7$), there exists a ground state solution to \eqref{eq.1.sch}. Moreover, the exact same arguments (with the obvious modifications) apply to problem \eqref{eq.s.sch} to show the existence of ground state solutions in the fractional case.

In order to apply our results we also need to impose some stronger assumptions on the potential function $V$, namely
\begin{enumerate}
\item[($V_2$)] $V\in L^\infty(\Omega)$ and there exists $\alpha_0>0$ such that $V(x)\ge \alpha_0$ a.e. $x\in\Omega$.
\end{enumerate}

An immediate consequence of ($V_2$) is that the norm $\|\cdot\|_s$ controls the Sobolev norm $\|\cdot\|_{s,p}$, i.e.
\begin{equation}\label{equiv.norm}
\|u\|_{s,p}^p \le \max\left\{1, \frac{1}{\alpha_0}\right\} \|u\|_s^p.
\end{equation}

Under these hypotheses, we get the following result.
\begin{thm}
Assume the same hypotheses of Theorem \ref{thm.ground} and moreover that $\Omega$ is unbounded and {\em ($f_7$), ($V_2$)} hold. Then, the same conclusions of Theorem \ref{thm.ground} hold true.
\end{thm}

\begin{proof}
Just observe that since ($V_2$) implies \eqref{equiv.norm}, all of the arguments in the proof of Theorem \ref{thm.ground} carry over to this case  without modifications.
\end{proof}

\section{Stability of eigenvalues}

In this section we consider the eigenvalue problem for the fractional $p-$Laplacian,
\begin{equation}\label{eq.s.eigen}
\begin{cases}
(-\Delta_p)^s u = \lambda^s |u|^{p-2}u & \text{ in }\Omega\\
u=0 & \text{ in } \R^n\setminus \Omega,
\end{cases}
\end{equation}
and its local counterpart
\begin{equation}\label{eq.1.eigen}
\begin{cases}
-\Delta_p u = \lambda^1 |u|^{p-2}u & \text{ in }\Omega\\
u=0 & \text{ on } \partial\Omega.
\end{cases}
\end{equation}

We will consider the bounded domain case.

For these problems it is known that there exists a sequence of {\em variational} eigenvalues $\{\lambda_k^s\}_{k\in\N}$ for each $s\in (0,1]$ given by the min-max formulation
\begin{equation}\label{lambdaks}
\lambda_k^s := \inf_{C\in {\mathcal C}_k^s} \max_{v\in C} \frac{[v]_{s,p}^p}{\|v\|_p^p},
\end{equation}
where  ${\mathcal C}_k^s$ denotes the compact, symmetric subsets of $W^{s,p}_0(\Omega)$ such that $\gamma(C)\ge k$, and $\gamma$ is the Krasnoselskii genus. See \cite{GAP} for $s=1$ and \cite{LL} in the fractional case $s\in (0,1)$. Of course in \eqref{lambdaks}, $[u]_{1,p}^p = \|\nabla u\|_p^p$.

This sequence of eigenvalues is denoted by $\Sigma_{\rm var}^s$. On the other hand, the spectrum of \eqref{eq.s.eigen} and \eqref{eq.1.eigen} is denoted by $\Sigma^s$, for $s\in (0,1]$.

Of course, $\Sigma_{\rm var}^s\subset \Sigma^s$ and a major open problem is to determine if equality holds.

The stability of the variational spectrum $\Sigma_{\rm var}^s$ as $s\uparrow 1$ was studied in \cite{BPS} and in that paper, the authors prove that $\lambda_k^s\to \lambda_k$ as $s\uparrow 1$ together with the convergence of the corresponding eigenfunctions. See \cite[Theorem 1.2]{BPS}.

For more stability results of different fractional eigenvalues problems, we refer the interested reader to \cite{FBSS}.

As an application of Theorem \ref{main} (more precisely, of Remark  \ref{remark.clave}) we obtain an stability result for eigenvalues of \eqref{eq.s.eigen} that gives much less information when applied to the variational sequence $\Sigma_{\rm var}^s$ but it can be applied to {\em any sequence of eigenvalues} $\lambda^s\in \Sigma^s$. Our result reads as follows:
\begin{thm}
Let $\lambda^s\in \Sigma^s$ be such that $\sup_{s\in (0,1)}\lambda^s <\infty$. Then any accumulation point $\lambda$ of the set $\{\lambda^s\}_{s\in(0,1)}$ belongs to $\Sigma^1$. Moreover, if $\{s_k\}_{k\in\N}$ is such that $s_k\to 1$ and $\lambda^{s_k}\to \lambda$ and $u_k\in W^{s_k,p}_0(\Omega)$ is an $L^p(\Omega)-$normalized eigenfunction of \eqref{eq.s.eigen} associated to $\lambda^{s_k}$, then, up to a further subsequence, there exists $u\in W^{1,p}_0(\Omega)$ such that $u_k\to u$ strongly in $L^p(\Omega)$ and $u$ is an $L^p(\Omega)-$normalized eigenfunction of \eqref{eq.1.eigen} associated to $\lambda$.
\end{thm}

\begin{proof}
The proof is an immediate consequence of Remark \ref{remark.clave}. In fact, assume that $\lambda^k := \lambda^{s_k}\to \lambda$ and let $u_k$ be the associated $L^p(\Omega)-$normalized eigenfunction of \eqref{eq.s.eigen}. Then, from \eqref{eq.s.eigen} one gets that
$$
[u_k]_{s_k,p}^p = \lambda^k \|u_k\|_p^p = \lambda^k\le C,
$$
with $C$ independent on $k\in\N$. Therefore, by \cite{BBM}, there exists $u\in W^{1,p}_0(\Omega)$ such that $u_k\to u$ strongly in $L^p(\Omega)$. Observe that this implies that $\|u\|_p=1$, so in particular, $u\neq 0$.

Now, since $f_k(z) := \lambda^k |z|^{p-2}z\to f(z) := \lambda |z|^{p-2}z$ uniformly on compact sets of $z\in\R$, from Remark \ref{remark.clave} it follows that $u$ is an eigenfunction of \eqref{eq.1.eigen} associated to $\lambda$ as we wanted to show.
\end{proof}

\section*{Acknowledgements}

This paper is partially supported by grants UBACyT 20020130100283BA, CONICET PIP 11220150100032CO and ANPCyT PICT 2012-0153. 

The authors are members of CONICET.

\bibliographystyle{amsplain}
\bibliography{biblio}

\end{document}